\UseRawInputEncoding
\documentclass[12pt, reqno]{amsart}
\usepackage{amssymb,latexsym,amsmath,amscd,amsfonts}
\usepackage{latexsym}
\usepackage[mathscr]{eucal}
\usepackage{bm}
\usepackage{mathptmx}

9.5in \numberwithin{equation}{section}

\newcommand{\beg}{\begin{equation}}
\newcommand{\eeg}{\end{equation}}
\newcommand{\ben}{\begin{eqnarray*}}
	\newcommand{\een}{\end{eqnarray*}}

\newtheorem{thm}{Theorem}[section]
\newtheorem{cor}[thm]{Corollary}
\newtheorem{lem}[thm]{Lemma}

\newtheorem{prop}[thm]{Proposition}
\numberwithin{equation}{section} 
\theoremstyle{definition}
\newtheorem{defn}[thm]{Definition}
\newtheorem{rem}[thm]{Remark}
\newtheorem{note}[thm]{Note}

\newcommand{\HS}{\mathcal H}
\newcommand{\C}{\mathbb{C}}

\newcommand{\D}{\mathbb{D}}
\newcommand{\T}{\mathbb{T}}
\newcommand{\ft}{\mathcal F_O}

\newcommand{\gn}{\mathbb{G}_n}

\newcommand{\ov}{\overline}

\begin{document}

\title[$C._0 \,\Gamma_n$-contractions]
{Dilation, functional model and a complete unitary invariant for $C._{0}\,\; \Gamma_n$-contractions}

\author[Sourav Pal]{Sourav Pal}
\address[Sourav Pal]{Mathematics Department, Indian Institute of Technology Bombay,
Powai, Mumbai - 400076, India.} \email{souravmaths@gmail.com , sourav@math.iitb.ac.in}



\thanks{The author was supported by the Seed Grant of IIT Bombay, the CPDA of Govt. of India, the
INSPIRE Faculty Award (Award No. DST/INSPIRE/04/2014/001462) of DST, India and the MATRICS
Award of SERB, (Award No. MTR/2019/001010) of DST, India.}

\begin{abstract}

 A commuting tuple of operators $(S_1,\dots, S_{n-1},P)$, defined on a Hilbert space $\mathcal H$, for which the closed symmetrized polydisc
\[
\Gamma_n =\left\{ \left(\sum_{1\leq i\leq n} z_i,\sum_{1\leq
i<j\leq n}z_iz_j,\dots, \prod_{i=1}^n z_i \right): \,|z_i|\leq 1,
i=1,\dots,n \right \}
\]
is a spectral set, is called a $\Gamma_n$-\textit{contraction}. A $\Gamma_n$-contraction is said to be
 \textit{pure} or $C._0$ if $P$ is $C._0$, that is, if ${P^*}^n \rightarrow 0$ strongly as $n \rightarrow \infty$. We show that for any $\Gamma_n$-contraction $(S_1,\dots, S_{n-1},P)$, there is a unique operator tuple $(A_1,\dots , A_{n-1})$ that satisfies the  operator identities
 \[
 S_i-S_{n-i}^*P=D_PA_iD_P\,, \quad \quad i=1,\dots, n-1.
 \]
This unique tuple is called the \textit{fundamental operator tuple} or $\ft$-tuple of $(S_1,\dots, S_{n-1},P)$. With the help of the $\ft$-tuple, we construct an operator model for a $C._0$ $\Gamma_n$-contraction and show that there exist $n-1$ operators $C_1,\dots, C_{n-1}$ such that each $S_i$ can be represented as $S_i=C_i+PC_{n-i}^*$. We find an explicit minimal dilation for a class of $C._0$ $\Gamma_n$-contractions whose $\ft$-tuples satisfy a certain condition. Also we establish that the $\ft$-tuple of $(S_1^*,\dots, S_{n-1}^*,P^*)$ together with the characteristic function of $P$ constitute a complete unitary invariant for the $C._0$ $\Gamma_n$-contractions. The entire program is an analogue of the Nagy-Foias theory for $C._0$ contractions.

\end{abstract}

\maketitle


\section{Introduction}

\noindent Throughout the paper all operators are bounded linear operators defined on complex Hilbert spaces. A contraction is an operator whose norm is not greater than $1$. We denote by $\C , \mathbb N$ the set of complex numbers and positive integers respectively. The open unit disc and the unit circle with center at the origin in $\C$ are denoted by $\D$ and $\T$ respectively.\\ 

For studying an operator, it suffices to consider only contractions because an operator is just a scalar multiple of a contraction. A few decades ago Sz.-Nagy and Foias initiated a program of determining the structure of a contraction and modeling it as a compression of a unitary operator. As a consequence numerous novel results were achieved and their beautiful constructive proofs were witnessed. A keen reader is referred to the classic \cite{nagy} and references there in. In 1951, von Neumann, \cite{von-Neumann}, introduced the notion of spectral set for an operator and described the contractions as operators having $\overline{\mathbb D}$ as a spectral set.
\begin{thm}[von Neumann, 1951]
An operator $T$ is a contraction if and only if $\overline{\mathbb
D}$ is a spectral set for $T$.
\end{thm}
The notion of spectral set, which was later defined by Arveson (e.g., \cite{arveson2}) for any finite tuple of commuting operators, became more popular and effective in deciphering the interplay between the intrinsic properties of an operator tuple and the complex geometry of an underlying compact subset of $\mathbb C^n$ associated with the tuple (e.g., see \cite{ AgMcC, paulsen}). In 1953, Sz.-Nagy published a very influential article, \cite{nagy1}, where he established the following dilation theorem whose impact is extraordinary till date.
\begin{thm}[Sz.-Nagy, 1953]
If $T$ is a contraction acting on a Hilbert space $\mathcal H$,
then there exists a Hilbert space $\mathcal K \supseteq \mathcal
H$ and a unitary $U$ on $\mathcal K$ such that
\[
p(T)=P_{\mathcal H}p(U)|_{\mathcal H} \quad \text{ for every polynomial } p\in \C[z].
\]
\end{thm}
A major and important class of operators that has been extensively studied by Sz.-Nagy, Foias and many other mathematicians (e.g., \cite{nagy}), is the $C._0$ class of contractions. A contraction $T$, defined on a Hilbert space $\mathcal H$, is said to be $ C._0$ or \textit{pure}, if ${T^*}^nh \rightarrow 0$ as $n \rightarrow \infty$ for all $h\in\mathcal H$. The aim of this article is to study the $C._0$ operator tuples associated with the symmetrized polydisc $\gn$, where
\[
\gn= \left\{ \left(\sum_{1\leq i\leq n} z_i,\sum_{1\leq
i<j\leq n}z_iz_j,\dots, \prod_{i=1}^n z_i \right): \,|z_i| < 1,\;
i=1,\dots,n \right \}.
\]
The symmetrized polydisc $\gn$ and its closure $\Gamma_n$, given by
\[
\Gamma_n= \left\{ \left(\sum_{1\leq i\leq n} z_i,\sum_{1\leq
i<j\leq n}z_iz_j,\dots, \prod_{i=1}^n z_i \right): \,|z_i|\leq 1, \;
i=1,\dots,n \right \},
\]
are the images of polydisc $\D^n$ and its closure $\ov{\D^n}$ respectively under the symmetrization map $\pi_n: \C^n \rightarrow \C^n$ defined by
\[
\pi_n(z_1,\dots, z_n) = \left(\sum_{1\leq i\leq n} z_i,\sum_{1\leq
i<j\leq n}z_iz_j,\dots, \prod_{i=1}^n z_i \right).
\]
The family of domains $\{ \gn :n\in \mathbb N \}$ was introduced in \cite{costara1} to study the spectral Nevanlinna-Pick problem. Indeed, the family of domains $\{ \mathbb G_n: n\in \mathbb N\}$ is naturally associated with spectral interpolation in the following way: if $\mathcal M_n(\mathbb C)$ is the space of all $n\times n$ complex matrices and if $\mathcal B_1$ is its spectral unit ball, then $A\in \mathcal B_1$ (that is, the spectral radius $r(A)<1$) if and only if $\Pi_n(A) \in \mathbb G_n$, where $\Pi_n(A)=\pi_n(\sigma(A))$, $\sigma(A)$ being the spectrum of $A$. The domain $\mathbb G_n$ is of importance because, apart from the derogatory matrices, the $n\times n$ spectral Nevanlinna-Pick problem is equivalent to a similar interpolation problem of $\mathbb G_n$ (see \cite{ay-ieot}, Theorem 2.1). Note that a bounded domain like $\mathbb G_n$, which has complex-dimension $n$, is much easier to deal with than an unbounded $n^2$-dimensional object like $\mathcal B_1$. The symmetrized polydisc has attracted considerable attention in past two decades because of its rich function theory, beautiful complex geometry and appealing operator theory (see \cite{costara1, edi-zwo, ay-jfa, ay-jot, tirtha-sourav, tirtha-sourav1, BSR} and references there in).\\

In this paper, we analyze and develop a Nagy-Foias type operator theory for the commuting operator tuples having $\Gamma_n$ as a spectral set.
\begin{defn}
A commuting tuple of operators $(S_1,\dots, S_{n-1},P)$
that has $\Gamma_n$ as a spectral set is called a $\Gamma_n$-\textit{contraction}, that is, $(S_1,\dots, S_{n-1},P)$ is a $\Gamma_n$-contraction if the Taylor joint spectrum $\sigma_T(S_1,\dots, S_{n-1},P) \subseteq \Gamma_n$ and von Neumann's inequality
\[
\|f(S_1,\dots, S_{n-1},P)\|\leq \sup_{\bold z \in \Gamma_n}|f(\bold z)|= \|f\|_{\infty, \,\Gamma_n} \qquad [\bold z=(s_1,\dots, s_{n-1},p)]\,,
\]
holds for all rational functions $f=p/q$, $(p, q \in \C[z_1,\dots ,z_n])$ such that $f$ does not have any pole in $\Gamma_n$. Also, a $\Gamma_n$-contraction $(S_1,\dots ,S_{n-1},P)$ is called \textit{pure} or $C._0$ if $P$ is a $C._0$ contraction.
\end{defn}
Here
$
f(S_1,\dots, S_{n-1},P)=p(S_1,\dots, S_{n-1},P)\,q(S_1,\dots, S_{n-1},P)^{-1}
$ 
in order to maintain the standard convention. In Section 2, we shall explain the motivation behind defining $C._0 \; \Gamma_n$-contraction in terms of the last component $P$. A Nagy-Foias type operator theoretic program for the pure $\Gamma_2$-contractions was initiated in \cite{tirtha-sourav} and was further carried out in \cite{tirtha-sourav1}. In this paper, we generalize those results for any arbitrary $n\geq 2$ and our methods are also a generalization of the techniques that were used in \cite{tirtha-sourav, tirtha-sourav1}. Operator theory on $\Gamma_2$ was simpler because rational dilation succeeded on $\Gamma_2$ (see \cite{tirtha-sourav}) and a concrete operator model was obtained as a consequence of dilation, \cite{tirtha-sourav1}. Since rational dilation fails on $\gn$ for $n\geq 3$ (see \cite{sourav1}), only conditional dilation and functional model can be achieved when $n\geq 3$.\\

In \cite{nagy}, Nagy and Foias showed that a $C._0$ contraction $T$, defined on a Hilbert space $\mathcal H$, can be realized as the compression of the shift operator $M_z$ defined on the vectorial Hardy space $H^2(\mathcal D_{T^*})$, where $\mathcal D_{T^*}= \overline{\text{Ran}}\, D_P= \overline{\text{Ran}}\,(I-TT^*)^{\frac{1}{2}}$. To obtain this representation, they first showed that $M_z$ on $H^2(\mathcal D_{P^*})$ is the minimal isometric dilation of $T$. In Theorem \ref{dilation-theorem}, we construct an analogous minimal $\Gamma_n$-isometric dilation of a $C._0 \;\Gamma_n$-contraction $(S_1,\dots ,S_{n-1},P)$ under certain conditions. The minimal dilation space for $(S_1,\dots ,S_{n-1},P)$ is no bigger than the Nagy-Foias minimal dilation space for $P$, which is $H^2(\mathcal D_{P^*})$. As a consequence of this $\Gamma_n$-isometric dilation, we achieve in Theorem \ref{modelthm} a concrete functional model for $(S_1,\dots ,S_{n-1},P)$. Also in Theorem \ref{fm}, we independently produce an operator model for a $C._0$ $\Gamma_n$-contraction without assuming any condition on $(S_1,\dots ,S_{n-1},P)$. This model may or may not be a commutative one.  A unique operator tuple $(B_1,\dots, B_{n-1})$ associated with $(S_1^*,\dots,S_{n-1}^*,P^*)$, which satisfies
\[
S_i^*-S_{n-i}P^*=D_{P^*}B_iD_{P^*} \qquad \; i=1, \dots ,n-1 ,
\]
plays the central role in the constructions of the dilation and the models. The existence and uniqueness of such an $(n-1)$-tuple are proved in Theorem \ref{existence-uniqueness}. Indeed, in Theorem \ref{existence-uniqueness} we consider the general case, that is, we prove that for every $\Gamma_n$-contraction $(S_1,\dots,S_{n-1},P)$ there exists a unique tuple $(A_1,\dots,A_{n-1})$ such that
\[
S_i-S_{n-i}^*P=D_PA_iD_P \qquad \qquad i=1,\dots ,n-1.
\]
For its pivotal role in operator theory on $\Gamma_n$, $(A_1,\dots, A_{n-1})$ is called the \textit{fundamental operator tuple} or shortly the $\ft$-\textit{tuple} of $(S_1,\dots,S_{n-1},P)$. In \cite{costara1}, Costara showed that for any $(s_1,\dots,s_{n-1},p)\in\Gamma_n$ there is a unique $(\beta_1,\dots,\beta_{n-1})\in \Gamma_{n-1}$ such that $ s_i=\beta_i+p\overline{\beta}_{n-i}$ for $i=1,\dots, n-1$.
With the help of the $\ft$-tuple, we find an operator theoretic analogue of this result for a $C._0$ $\Gamma_n$-contraction in Corollaries \ref{representation}, \ref{representation1}. In view of the existence and uniqueness of the $\ft$-tuple of a $\Gamma_n$-contraction (as in Theorem \ref{existence-uniqueness}), a natural question arises: given $n-1$ operators $A_1,\dots, A_{n-1}$, can we find a $\Gamma_n$-contraction for which $(A_1,\dots, A_{n-1})$ is the $\ft$ -tuple ? We provide a partial answer to this question in Theorem \ref{converse}.

One of the most wonderful discoveries in operator theory is the characteristic function of a contraction due to Nagy and Foias (defined in Subsection \ref{op-model}) which is a complete unitary invariant for the completely non-unitary (c.n.u) contractions in the sense that two c.n.u contractions $T_1,T_2$ are unitarily equivalent if and only if their characteristic functions coincide. In Theorem \ref{unitary inv}, we find a complete unitary invariant
for the $C._0$ $\Gamma_n$-contractions. We show that for a $C._0$ $\Gamma_n$-contraction $(S_1,\dots,S_{n-1},P)$,
the $\ft$-tuple $(B_1,\dots,B_{n-1})$ of $(S_1^*,\dots, S_{n-1}^*,P^*)$ and the characteristic function $\Theta_P$ of $P$ constitute a complete unitary invariant.\\

In Section 2, we accumulate a few definitions and results from the literature which will be used in sequel.\\

\noindent \textbf{Note.} The present article is an updated version of a part of the author's unpublished paper \cite{sourav12}. We learned that Theorem \ref{existence-uniqueness} and the preparatory result Proposition \ref{dilation-extension} of this article were independently proved by A. Pal in \cite{A:P}.\\

\noindent \textbf{\textit{Acknowledgement.}} The author is thankful to Orr Shalit for several fruitful comments on this article.

\section{A brief literature and preliminaries}

\vspace{0.4cm}

\noindent Unitary, isometry and co-isometry are special classes of contractions. There are natural analogues of these classes for $\Gamma_n$-contractions in the literature (see \cite{ay-jot, BSR, sourav1}). It was established in \cite{edi-zwo} that the \textit{distinguished boundary} of $\Gamma_{n}$, denoted by $b\Gamma_{n}$, is the symmetrization of the distinguished boundary of the polydisc, which is the $n$-torus $\mathbb T^n$, and thus $b\Gamma_n$ is the set
\[
b\Gamma_{n}=\left\{\left(\sum\limits_{i=1}^{n}z_i, \sum\limits_{1\leq i<j\leq n}z_iz_j, \dots,\prod_{i=1}^nz_i\right): |z_i|=1,\, i=1, \dots, n\right\}.
\]

\begin{defn}
Let $S_1,\dots, S_{n-1},P$ be commuting operators on $\mathcal H$.
Then $(S_1,\dots, S_{n-1},P)$ is called
\begin{itemize}
\item[(i)] a $\Gamma_n$-\textit{unitary} if $S_1,\dots, S_{n-1},P$
are normal operators and $\sigma_T
(S_1,\dots, S_{n-1},P) \subseteq b\Gamma_n$ ;

\item[(ii)] a $\Gamma_n$-\textit{isometry} if there exist a Hilbert space
$\mathcal K \supseteq \mathcal H$ and a $\Gamma_n$-unitary
$(T_1,\dots,T_{n-1},U)$ on $\mathcal K$ such that $\mathcal H$ is
a joint invariant subspace of $T_1,\dots, T_{n-1},U$ and
\[
(T_1|_{\mathcal H},\dots, T_{n-1}|_{\mathcal H},U|_{\mathcal
H})=(S_1,\dots, S_{n-1},P) ;
\]

\item[(iii)] a $\Gamma_n$-\textit{co-isometry} if the adjoint
$(S_1^*,\dots, S_{n-1}^*,P^*)$ is a $\Gamma_n$-isometry.
\end{itemize}
\end{defn}

The following theorems from \cite{sourav1} provide clear
descriptions of a $\Gamma_n$-unitary and a $\Gamma_n$-isometry.

\begin{thm}[\cite{sourav1}, Theorems 4.2 $\&$
4.4]\label{thm:gamma-ui} A commuting tuple of operators
$(S_1,\dots,S_{n-1},P)$ is a $\Gamma_n$-unitary $($or, a
$\Gamma_n$-isometry$)$ if and only if $(S_1,\dots,S_{n-1},P)$ is a
$\Gamma_n$-contraction and $P$ is a unitary $($isometry$)$.
\end{thm}
Needless to mention that $(S_1,\dots,S_{n-1},P)$ is a
$\Gamma_n$-co-isometry if and only if $(S_1,\dots,S_{n-1},P)$ is a
$\Gamma_n$-contraction and $P$ is a co-isometry. So, it is evident
that the nature of a $\Gamma_n$-contraction
$(S_1,\dots,S_{n-1},P)$ is highly influenced by the nature of its
last component $P$. In \cite{sourav3}, the author of
this paper showed that for a given $\Gamma_n$-contraction
$(S_1,\dots,S_{n-1},P)$ on $\mathcal H$, if $P=P|_{\mathcal
H_1}\oplus P|_{\mathcal H_2}$ is the canonical decomposition of
the contraction $P$ with respect to $\mathcal H= \mathcal H_1 \oplus \mathcal H_2$ such that $P|_{\HS_1}$ is a unitary and $P|_{\HS_2}$ is a c.n.u contraction, then both
$\mathcal H_1, \mathcal H_2$ reduce $S_1,\dots,S_{n-1}$ and
$(S_1|_{\mathcal H_1},\dots,S_{n-1}|_{\mathcal H_1},P|_{\mathcal
H_1})$ is a $\Gamma_n$-unitary whereas $(S_1|_{\mathcal
H_2},\dots,S_{n-1}|_{\mathcal H_2},P|_{\mathcal H_2})$ is a
$\Gamma_n$-contraction for which $P|_{\mathcal H_2}$ is a c.n.u
contraction. This unique decomposition was named the
``\textit{canonical decomposition}" of a $\Gamma_n$-contraction. This naturally motivated the author to define a c.n.u $\Gamma_n$-contraction to
be a $\Gamma_n$-contraction $(S_1,\dots, S_{n-1},P)$ for which $P$
is a c.n.u contraction and indeed such a definition is justified.
Taking cue from such dominant roles of $P$ in determining the
special classes of a $\Gamma_n$-contraction
$(S_1,\dots,S_{n-1},P)$ we are led to the following definition.
\begin{defn}\label{def2.5}
A $\Gamma_n$-contraction $(S_1, \dots, S_{n-1}, P)$ acting on $\HS$ is said to be $C._0$ or \textit{pure} if $P$ is a $C._0$ contraction, that is, if ${P^*}^n h \rightarrow 0 $ as $n \rightarrow \infty $ for all $h\in \HS$.
\end{defn}
The following theorem, which provides a characterization for the $\Gamma_n$-unitaries, will be used in sequel.

\begin{thm}[\cite{BSR}, Theorem 4.2]\label{thm:tu}
Let $(S_1,\dots, S_{n-1}, P)$ be a commuting tuple of bounded
operators. Then the following are equivalent.

\begin{enumerate}

\item $(S_1,\dots,S_{n-1},P)$ is a $\Gamma_n$-unitary,

\item $P$ is a unitary,
$(\frac{n-1}{n}S_1,\frac{n-2}{n}S_2,\dots,\frac{1}{n}S_{n-1})$ is
a $\Gamma_{n-1}$-contraction and $S_i = S_{n-i}^* P$ for
$i=1,\dots,n-1$.
\end{enumerate}
\end{thm}

\section{The fundamental operator tuple ($\mathcal F_O$-tuple) of a $\Gamma_n$
-contraction}

\vspace{0.4cm}

\noindent In \cite{sourav3}, we introduced the following $n-1$
operator pencils $\Phi_1,\dots,\Phi_{n-1}$ in order to determine
the structure of a $\Gamma_n$-contraction $(S_1,\dots,S_{n-1},P)$:
\begin{align}
\Phi_{i}(S_1,\dots, S_{n-1},P) &= (\tilde n_i-S_i)^*(\tilde
n_i-S_i)-(\tilde n_i P-S_{n-i})^*(\tilde n_i P-S_{n-i}) \notag
\\&
={\tilde n_i}^2(I-P^*P)+(S_i^*S_i-S_{n-i}^*S_{n-i})-\tilde n_i(S_i-S_{n-i}^*P) \notag \\
& \quad \quad -\tilde n_i(S_i^*-P^*S_{n-i})\,, \quad \quad \text{
where } \tilde n_i= \binom{n}{i} \label{eq:1a}.
\end{align}
We mention here to the readers that while defining $\Phi_i$ in
\cite{sourav3}, $\tilde n_i$ was mistakenly displayed as $n$ and
that was a typographical error. From the definition it is clear
that in particular when $S_1,\dots,S_{n-1}, P$ are scalars, i.e,
points in $\Gamma_n$, the above operator pencils take the
following form for each $i$ :
\begin{align}
\Phi_{i}(s_1,\dots,s_{n-1},p) & = {\tilde
n_i}^2(1-|p|^2)+(|s_i|^2-|s_{n-i}|^2)-\tilde
n_i(s_i-\bar{s}_{n-i}p) \notag \\ & \quad \quad -\tilde
n_i(\bar{s}_i-\bar{p}s_{n-i}). \label{eqn:2a}
\end{align}
The following result appeared in \cite{sourav3} and is extremely
important in the context of this paper.

\begin{prop}[Proposition 2.6, \cite{sourav3}]\label{lem:3}
Let $(S_1,\dots,S_{n-1},P)$ be a $\Gamma_n$-contraction. Then for
$i=1,\dots,n-1,\; \Phi_i(\alpha
S_1,\dots,\alpha^{n-1}S_{n-1},\alpha^n P)\geq 0$ for all $\alpha
\in\overline{\mathbb D}$.
\end{prop}

The positivity of the operator pencils $\Phi_i$ will determine a
certain and unique operator tuple $(A_1,\dots, A_{n-1})$
associated with each $\Gamma_n$-contraction $(S_1,\dots,
S_{n-1},P)$. We shall call $(A_1,\dots, A_{n-1})$ \textit{the
fundamental operator tuple} or the $\ft$-\textit{tuple} of $(S_1,\dots, S_{n-1},P)$ and the underlying reason is that it plays the central role in constructing the explicit dilation and operator models and in determining the complete unitary invariant for a $C._0 \;\Gamma_n$-contraction.\\

Recall that the {\em numerical radius} of an operator $A$
on a Hilbert space $\mathcal{H}$ is defined by
\[
\omega(A) = \sup \{|\langle Ax,x \rangle|\; : \;
\|x\|_{\mathcal{H}}= 1\}.
\]
It is well known that
\begin{eqnarray}\label{nradius}
r(A)\leq \omega(A)\leq \|A\| \textup{ and } \frac{1}{2}\|A\|\leq
\omega(A)\leq \|A\|, \end{eqnarray} where $r(A)$ is the spectral
radius of $A$. We state a basic lemma on the numerical radius whose proof is a routine exercise. We shall use this lemma in sequel.

\begin{lem} \label{basicnrlemma}
The numerical radius of an operator $A$ is not greater than $1$ if
and only if  Re $(\alpha A) \leq I$ for all complex numbers
$\alpha$ of modulus $1$.
\end{lem}

\begin{thm}{\bf (Existence and
Uniqueness).}\label{existence-uniqueness} Let
$(S_1,\dots,S_{n-1},P)$ be a $\Gamma_n$-contraction on a Hilbert
space $\mathcal H$. Then there are unique operators
$A_1,\dots,A_{n-1}\in\mathcal B(\mathcal D_P)$ such that
\[
S_i-S_{n-i}^*P=D_PA_iD_P \text{ for } i=1,\dots,n-1.
\]
Moreover, for each $i$ and for all $z\in \mathbb T$, $\omega
(A_i+A_{n-i}z)\leq \tilde n_i$. $[\tilde n_i= \binom{n}{i}]$.
\end{thm}

\begin{proof}
We apply Proposition \ref{lem:3} to $(S_1,\dots,S_{n-1},P)$ and
obtain for each $i=1,\dots, n-1$,
\begin{equation}
\Phi_{i}(\alpha S_1,\dots,\alpha^{n-1}S_{n-1},\alpha^nP)\geq 0
\,,\label{eqn11}
\end{equation}
for all $\alpha\in\overline{\mathbb D}$. Therefore, in particular
for $\beta, \gamma \in\mathbb T$ we have from (\ref{eqn11}) for $\Phi_i$ and $\Phi_{n-i}$ respectively
\begin{align}
& \label{10f} {\tilde n_i}^2(I-P^*P)+(S_i^*S_i-S_{n-i}^*S_{n-i}) \geq \tilde
n_i\beta^i(S_i-S_{n-i}^*P)+\tilde n_i
\bar{\beta}^i(S_{i}^*-P^*S_{n-i}) \,,
\\&
\label{11f}
{\tilde n_i}^2(I-P^*P)+(S_{n-i}^*S_{n-i}-S_{i}^*S_i)\geq \tilde
n_i\gamma^{n-i} (S_{n-i}-S_i^*P)+\tilde
n_i\bar{\gamma}^{n-i}(S_{n-i}^*-P^*S_i).
\end{align}
We choose $\beta, \gamma \in \mathbb T$ such that $\gamma^{n-i}=\beta^i=\eta$ and then by adding we get
\begin{align}
2\tilde n_i(I-P^*P) & \geq
\eta \{(S_i-S_{n-i}^*P)+(S_{n-i}-S_i^*P) \} \notag \\&
\label{10a} \quad + \bar{\eta} \{(S_i^*-P^*S_{n-i})+
(S_{n-i}^*-P^*S_i)\}.
\end{align}
This shows that the Laurent polynomial
\begin{align}
\xi(z)=& 2\tilde
n_i(I-P^*P)-z\{(S_i-S_{n-i}^*P)+(S_{n-i}-S_i^*P) \} \notag
\\& \label{11a} -\bar{z}\{(S_i^*-P^*S_{n-i})+
(S_{n-i}^*-P^*S_i)\}
\end{align}
is non-negative for all $z\in\mathbb T$. Therefore, by the Operator
Fejer-Riesz Theorem (see Theorem 1.2 in \cite{DW}) there is a
polynomial of degree $1$, say $P(z)=X_0+X_1z$, such that for
all $z\in\mathbb T$,
\begin{align}
\xi(z)=P(z)^*P(z)
&=(X_0^*+X_1^*\bar{z})(X_0+X_1z)\notag \\&=
(X_0^*X_0+X_1^*X_1)+(X_0^*X_1)z+X_1^*X_0\bar z. \label{11b}
\end{align}
Comparing (\ref{11a}) and (\ref{11b}) we obtain
\begin{gather}
\label{41} 2\tilde n_i D_P^2 =X_0^*X_0+X_1^*X_1 \\
\label{42} (S_i-S_{n-i}^*P)+(S_{n-i}-S_i^*P) =-X_0^*X_1.
\end{gather}
Again putting $\gamma=-\beta$ in (\ref{11f}) and adding (\ref{10f}) and (\ref{11f}) we have that
\begin{align}
2\tilde n_i(I-P^*P) & \geq
\beta \{(S_i-S_{n-i}^*P)-(S_{n-i}-S_i^*P) \} \notag \\&
\label{10e} \quad + \bar{\beta} \{(S_i^*-P^*S_{n-i})-
(S_{n-i}^*-P^*S_i)\}.
\end{align}
This shows that the Laurent polynomial
\begin{align}
\eta(z)=& 2\tilde
n_i(I-P^*P)-z\{(S_i-S_{n-i}^*P)-(S_{n-i}-S_i^*P) \} \notag
\\& \label{11e} -\bar{z}\{(S_i^*-P^*S_{n-i})-
(S_{n-i}^*-P^*S_i)\}
\end{align}
is non-negative for all $z\in\mathbb T$.
Therefore, applying the Operator
Fejer-Riesz Theorem again we have a
polynomial of degree $1$, say $Q(z)=Y_0+Y_1z$, such that for
all $z\in\mathbb T$,
\begin{align}
\eta(z)=Q(z)^*Q(z)
&=(Y_0^*+Y_1^*\bar{z})(Y_0+Y_1z)\notag \\&=
(Y_0^*Y_0+Y_1^*Y_1)+(Y_0^*Y_1)z+Y_1^*Y_0\bar z. \label{11g}
\end{align}
Comparing (\ref{11e}) and (\ref{11g}) we have
\begin{gather}
\label{41e} 2\tilde n_i D_P^2 =Y_0^*Y_0+Y_1^*Y_1 \\
\label{42e} (S_i-S_{n-i}^*P)-(S_{n-i}-S_i^*P) =-Y_0^*Y_1.
\end{gather}
Adding (\ref{41}) and (\ref{41e}) we have
\begin{equation}\label{41h}
4\tilde n_iD_P^2=(X_0^*X_0+X_1^*X_1)+(Y_0^*Y_0+Y_1^*Y_1).
\end{equation}
Similarly adding (\ref{42}) and (\ref{42e}) we have
\begin{equation}\label{42h}
2(S_i-S_{n-i}^*P)=-(X_0^*X_1+Y_0^*Y_1).
\end{equation}
We obtain from (\ref{41h}) that
\[
4\tilde n_i D_P^2\geq X_0^*X_0, \; 4\tilde n_i D_P^2\geq X_1^*X_1,\;
4\tilde n_i D_P^2\geq Y_0^*Y_0\text{ and } 4\tilde n_i D_P^2\geq Y_1^*Y_1\,.
\]
So, by Douglas's lemma (see Lemma 2.1 in \cite{DMP}) there
are contractions $Z_0,Z_1,Z_2, Z_3$ such that
\begin{align*}
X_0^* & =2\sqrt{\tilde n_i}D_PZ_0,\quad  X_1^*=2\sqrt{\tilde n_i}D_PZ_1,\\
 Y_0^* &=2\sqrt{\tilde n_i}D_PZ_2, \quad Y_1^*=2\sqrt{\tilde n_i}D_PZ_3.
\end{align*}

Substituting these values in (\ref{42h}) we have that
\[
S_i-S_{n-i}^*P =D_P[-2\tilde n_i(Z_0Z_1^*+Z_2Z_3^*)]D_P.
\]
Setting $A_i=P_{\mathcal D_P}[-2\tilde n_i(Z_0Z_1^*+Z_2Z_3^*)]|_{\mathcal D_P}$, we have that
\[
S_i-S_{n-i}^*P=D_PA_iD_P\,,
\]
which is true for all $i=1,\dots ,n-1$.\\

\noindent Adding (\ref{10a}) and (\ref{10e}) we have that
\[
4\tilde n_i D_P^2 \geq 4 \text{ Re }\beta (S_i-S_{n-i}^*P)
 =4 \text{ Re }\beta D_PA_iD_P.
\]
Therefore,
\[
D_P^2\geq \text{ Re } \beta D_P(\dfrac{1}{\tilde n_i}A_i)D_P\,,
\]
and thus
\[
D_P[I_{\mathcal D_P}-\text{Re }\beta (\dfrac{1}{\tilde n_i}A_i)]D_P \geq 0\,.
\]
This implies that
\[
I_{\mathcal D_P}-\text{Re }\beta (\dfrac{1}{\tilde n_i}A_i) \geq 0
\]
because $A_i$ is defined on $\mathcal D_P$. Therefore, by Lemma
\ref{basicnrlemma}, we have
\[
\omega(A_i)\leq \tilde n_i= \binom{n}{i}\,.
\]

\vspace{0.4cm}

\noindent {\bf Uniqueness.} Let there be two solutions $A_i,C_i$
of the equation $S_i-S_{n-i}^*P=D_PX_iD_P$. Then
$D_P(A_i-C_i)D_P=0$, which shows that $A_i-C_i=0$ as $A_i-C_i$ is
defined on $\mathcal D_P$. Thus $A_i$ is unique for each $i=1,\dots, n-1$.

\end{proof}

\noindent We shall see in the next subsection a partial converse
to the existence-uniqueness of the $\ft$-tuple. Under a certain
condition, an operator tuple $(A_1,\dots,A_{n-1})$ becomes the
$\ft$-tuple of a $\Gamma_n$-contraction (see Theorem
\ref{converse}).

\begin{note}
The $\ft$-tuple of a $\Gamma_n$-isometry or a $\Gamma_n$-unitary
$(S_1,\dots,S_{n-1},P)$ is defined to be $(0,\dots,0)$ because the
$\ft$-tuple is defined on the space $\mathcal D_P$ and in such
cases $\mathcal D_P=\{0\}$.
\end{note}

\begin{prop}\label{end-prop}
If two $\Gamma_n$-contractions are unitarily equivalent then so
are their $\ft$-tuples.
\end{prop}
\begin{proof}
Suppose $(S_{11},\dots,S_{1(n-1)},P_1)$ and
$(S_{21},\dots,S_{2(n-1)},P_2)$ are two unitarily equivalent
$\Gamma_n$-contractions acting on Hilbert spaces $\mathcal H_1$ and
$\mathcal H_2$ respectively with $\ft$-tuples $(F_1,\dots,F_{n-1})$ and
$(G_1,\dots,G_{n-1})$. Then there is a unitary $U$ from $\mathcal
H_1$ to $\mathcal H_2$ such that
\[
US_{11}=S_{21}U \;,\dots , US_{1(n-1)}=S_{2(n-1)}U \text{ and }\;
UP_1=P_2U\,.
\]
Obviously $UP_1^*=P_2^*U$ and consequently
\[
UD_{P_1}^2=U(I-P_1^*P_1)=U-P_2^*P_2U=D_{P_2}^2U\,.
\]
Therefore, $UD_{P_1}=D_{P_2}U$. Let $V=U|_{\mathcal D_{P_1}}$.
Then $V\in\mathcal L(\mathcal D_{P_1},\mathcal D_{P_2})$ and
$VD_{P_1}=D_{P_2}V$. Thus, using the fact that
$S_{1i}-S_{1(n-i)}^*P_1$ and $S_{2i}-S_{2(n-i)}^*P_2$ are equal to
$0$ on the orthogonal complement of $\mathcal D_{P_1}$ and
$\mathcal D_{P_2}$ respectively we have
\[
D_{P_2}VF_iV^*D_{P_2} = VD_{P_1}F_iD_{P_1}V^* =
V(S_{1i}-S_{1(n-i)}^*P_1)V^* = S_{2i}-S_{2{n-i}}^*P_2
=D_{P_2}G_iD_{P_2}\,.
\]
So, $F_i$ and $G_i$ are unitarily equivalent for $i=1,\dots,n-1$ and the proof is complete.

\end{proof}

\begin{rem}
The converse to the above result does not hold, i.e, two
non-unitarily equivalent $\Gamma_n$-contractions can have
unitarily equivalent $\ft$-tuples. For example, if we consider a
$\Gamma_n$-isometry on a Hilbert space which is not a
$\Gamma_n$-unitary, then its $\ft$-tuple is $(0,\dots,0)$ which is
same as the $\ft$-tuple of any $\Gamma_n$-unitary on the same
Hilbert space.
\end{rem}

\noindent \textbf{A partial converse to the Existence-Uniqueness Theorem for the $\ft$-tuple.} The existence and uniqueness of the $\ft$-tuple (Theorem
\ref{existence-uniqueness}) is in the center of all results of
this article. Here we provide a partial converse to that result.

\begin{thm}\label{converse}
Let $A_1,\dots,A_{n-1}$ be operators defined on a Hilbert space
$E$ such that
$$
\left(
\frac{n-1}{n}(A_1^*+A_{n-1}z),\frac{n-2}{n}(A_2^*+A_{n-2}z),\dots,\frac{1}{n}(A_{n-1}^*+A_1z)
\right)
$$
is a $\Gamma_{n-1}$-contraction for all $z\in\mathbb T$. Then
there is a $\Gamma_n$-contraction for which $(A_1,\dots,A_{n-1})$
is the $\ft$-tuple.
\end{thm}

\begin{proof}
Let us consider the vectorial Hardy-Hilbert space $H^2(E)$ and the
Toeplitz operator tuple
$(T_{A_1^*+A_{n-1}z},\dots,T_{A_{n-1}^*+A_1z},T_z)$ acting on it.
Here $T_z$ on $H^2(E)$ is the shift operator. We shall show that
$
(T_{A_1^*+A_{n-1}z}^*,\dots,T_{A_{n-1}^*+A_1z}^*,T_z^*)
$
is a $\Gamma_n$-co-isometry and $(A_1,\dots,A_{n-1})$ is the
$\ft$-tuple of it. Since for all $z \in \mathbb T$, $$\left(
\frac{n-1}{n}(A_1^*+A_{n-1}z),\frac{n-2}{n}(A_2^*+A_{n-2}z),\dots,
\frac{1}{n}(A_{n-1}^*+A_1z) \right)$$ is a
$\Gamma_{n-1}$-contraction, so is $$\left(
\frac{n-1}{n}M_{A_1^*+A_{n-1}z},\frac{n-2}{n}M_{A_2^*+A_{n-2}z},
\dots,\frac{1}{n}M_{A_{n-1}^*+A_1z} \right),$$ where each of the
multiplication operators is defined on $L^2(E)$. It is evident
that for each $z\in\mathbb T$, $M_{{A_i}^*+{A_{n-i}}z}=M_{{A_{n-i}}^*+{A_i}z}^*M_z $ and $M_z$
on $L^2(E)$ is unitary. So by Theorem \ref{thm:tu}, the multiplication operator tuple $(M_{A_1^*+A_{n-1}z},\dots,M_{A_{n-1}^*+A_1z},M_z)$ on $L^2(E)$ is a $\Gamma_n$-unitary and the Toeplitz operator tuple
$(T_{A_1^*+A_{n-1}z},\dots,T_{A_{n-1}^*+A_1z},T_z)$, being the
restriction of $(M_{A_1^*+A_{n-1}z},\dots,M_{A_{n-1}^*+A_1z},M_z)$
to the common invariant subspace $H^2(E)$, is a
$\Gamma_n$-isometry. Therefore,
$
(T_{A_1^*+A_{n-1}z}^*,\dots,T_{A_{n-1}^*+A_1z}^*,T_z^*)
$
is a $\Gamma_n$-co-isometry. We now compute the $\ft$-tuple of
$(T_{A_1^*+A_{n-1}z}^*,\dots,T_{A_{n-1}^*+A_1z}^*,T_z^*). $ Now
for each $i=1,\dots,n-1$,
\[
T_{A_i^*+A_{n-i}z}^*- T_{A_{n-i}^*+A_iz}T_z^*=T_{A_i+A_{n-i}^*
\bar z}- T_{A_{n-i}^*+A_iz}T_{\bar z}= T_{A_i}=A_i.
\]
Again since $I-T_zT_z^*$ is the projection onto the space
$\mathcal D_{T_z^*} (=E)$,
\[
(I-T_zT_z^*)^{\frac{1}{2}}A_i(I-T_zT_z^*)^{\frac{1}{2}}={A_i}
\]
Therefore, by the uniqueness of the $\ft$-tuple, $(A_1,\dots,A_{n-1})$
is the $\ft$-tuple of the $\Gamma_n$-co-isometry
$(T_{A_1^*+A_{n-1}z}^*,\dots,T_{A_{n-1}^*+A_1z}^*,T_z^*)$.

\end{proof}

\section{Dilation and model theory for $C._0$ $\Gamma_n$-contractions}

\vspace{0.4cm}

\subsection{Dilation}

We have witnessed in \cite{sourav1} that in general rational
dilation fails on the symmetrized polydisc in any dimension
greater than $2$. In this subsection, we shall determine a class
of $C._0$ $\Gamma_n$-contractions that dilate to the distinguished
boundary $b\Gamma_n$. Indeed, we impose certain conditions on the
$\ft$-tuple of a $C._0$ $\Gamma_n$-contraction to obtain a
$\Gamma_n$-isometric dilation and then extend that $\Gamma_n$-isometry to a $\Gamma_n$-unitary which eventually becomes a
$\Gamma_n$-unitary dilation.

\begin{defn}
Let $(S_1,\dots,S_{n-1},P)$ be a $\Gamma_n$-contraction on
$\mathcal H$. A commuting operator tuple $(T_1,\dots,T_{n-1},V)$ defined on
$\mathcal K$ is said to be a $\Gamma_n$-isometric dilation of
$(S_1,\dots,S_{n-1},P)$ if $\mathcal H \subseteq \mathcal K$,
$(T_1,\dots,T_{n-1},V)$ is a $\Gamma_n$-isometry and
\[
P_{\mathcal H}(T_1^{m_1}\dots T_{n-1}^{m_{n-1}}V^{n})|_{\mathcal
H} =S_1^{m_1}\dots S_{n-1}^{m_{n-1}}P^{n},
\]
for all non-negative integers $ m_1, \dots, m_{n-1},n.$ Moreover,
the dilation is called {\em minimal} if the following holds:
\[
\mathcal K=\overline{\textup{span}}\{ T_1^{m_1}\dots
T_{n-1}^{m_{n-1}}V^n h\,:\; h\in\mathcal H \textup{ and
}m_1,\dots, m_{n-1},n\in \mathbb N \cup \{0\} \}.
\]
In a similar fashion one can define a $\Gamma_n$-unitary dilation of
a $\Gamma_n$-contraction.
\end{defn}
Needless to mention that a $\Gamma_n$-unitary dilation of a
$\Gamma_n$-contraction is a normal dilation to the distinguished
boundary $b\Gamma_n$. Before going to the explicit dilation, we
state and prove the following result which we use in the proof of
the dilation theorem.

\begin{prop}\label{dilation-extension}
Let $(T_1,\dots,T_{n-1},V)$ on $\mathcal K$ be a
$\Gamma_n$-isometric dilation of a $\Gamma_n$-contraction
$(S_1,\dots,S_{n-1},P)$ on $\mathcal H$. If
$(T_1,\dots,T_{n-1},V)$ is minimal, then
$(T_1^*,\dots,T_{n-1}^*,V^*)$ is a $\Gamma_n$-co-isometric
extension of $(S_1^*,\dots,S_{n-1}^*,P^*)$. Conversely, the
adjoint of a $\Gamma_n$-co-isometric extension of
$(S_1,\dots,S_{n-1},P)$ is a $\Gamma_n$-isometric dilation of
$(S_1,\dots,S_{n-1},P)$.
\end{prop}
\begin{proof}
We first prove that $S_iP_{\mathcal H}=P_{\mathcal H}T_i$ for each
$i$ and $PP_{\mathcal H}=P_{\mathcal H}V$. Clearly
$$\mathcal K=\overline{\textup{span}}\{ T_1^{m_1}\dots T_{n-1}^{m_{n-1}}V^n h\,:\;
h\in\mathcal H \textup{ and }m_1,\dots,m_{n-1},n\in \mathbb N \cup
\{0\} \}.$$ Now for $h\in\mathcal H$ we have that
\begin{align*}
S_iP_{\mathcal H}(T_1^{m_1}\dots T_{n-1}^{m_{n-1}}V^n h) &
=S_i(S_1^{m_1}\dots S_{n-1}^{m_{n-1}}P^n h) \\ & =S_1^{m_1}\dots
S_i^{m_i+1}\dots S_{n-1}^{m_{n-1}}P^n h \\& =P_{\mathcal
H}(T_1^{m_1}\dots T_i^{m_i+1}\dots T_{n-1}^{m_{n-1}}V^n h)\\&
=P_{\mathcal H}T_i(T_1^{m_1}\dots T_{n-1}^{m_{n-1}}V^n h).
\end{align*}
Thus, $S_iP_{\mathcal H}=P_{\mathcal H}T_i$. Similarly we can
prove that $PP_{\mathcal H}=P_{\mathcal H}V$. Also for
$h\in\mathcal H$ and $k\in\mathcal K$ we have that
\[
\langle S_i^*h,k \rangle =\langle P_{\mathcal H}S_i^*h,k \rangle
=\langle S_i^*h,P_{\mathcal H}k \rangle  =\langle h,S_iP_{\mathcal
H}k \rangle =\langle h,P_{\mathcal H}T_ik \rangle =\langle
T_i^*h,k \rangle .
\]
Hence $S_i^*=T_i^*|_{\mathcal H}$ and similarly
$P^*=V^*|_{\mathcal H}$. Therefore, $(T_1^*,\dots,T_{n-1}^*,V^*)$
is a $\Gamma_n$-co-isometric extension of
$(S_1^*,\dots,S_{n-1}^*,P^*)$. The converse part is obvious.

\end{proof}

We are now in a position to present the desired dilation theorem which is one of the main results of this article.

\begin{thm} \label{dilation-theorem}
Let $(S_1,\dots , S_{n-1},P)$ be a $C._0$ $\Gamma_n$-contraction
on a Hilbert space $\mathcal{H}$ and let the $\ft$-tuple
$(B_1,\dots,B_{n-1})$ of $(S_1^*,\dots,S_{n-1}^*,P^*)$ be such
that $$\left(
\frac{n-1}{n}(B_{1}^*+B_{n-1}z),\frac{n-2}{n}(B_{2}^*+B_{n-2}z),
\dots, \frac{1}{n}(B_{n-1}^*+B_{1}z) \right)$$ is a
$\Gamma_{n-1}$-contraction for all $z \in \mathbb T$. Consider the
operators $T_1,\dots,T_{n-1},V$ on $\mathcal{K}=H^2(\mathbb{D})
\otimes \mathcal{D}_{P^*}$ defined by
$
T_i=I\otimes B_{i}^*+M_z\otimes B_{n-i}$ for $i=1,\dots,
n-1$ and $V=M_z\otimes I$. Then the $n$-tuple $(T_1,\dots,T_{n-1},V)$ is a minimal $C._0 \; \Gamma_n$-isometric dilation of $(S_1,\dots,S_{n-1},P)$.
\end{thm}

\begin{proof}
We have from the Sz.-Nagy-Foias theory for $C._0$ contractions
\cite{nagy} that $\mathcal K$ is the minimal isometric dilation
space and the operator $V$ is the minimal isometric dilation of
$P$. So, the minimality of the dilation follows trivially. It
suffices to prove that $(T_1,\dots,T_{n-1},V)$ is a
$\Gamma_n$-isometric dilation of $(S_1,\dots,S_{n-1},P)$. By
virtue of Lemma \ref{dilation-extension}, it suffices to show
that $(T_1^*,\dots,T_{n-1}^*,V^*)$ is a $\Gamma_n$-co-isometric
extension of $(S_1^*,\dots,S_{n-1}^*,P^*)$. Since
\[
\left(
\frac{n-1}{n}(B_{1}^*+B_{n-1}z),\frac{n-2}{n}(B_{2}^*+B_{n-2}z),
\dots, \frac{1}{n}(B_{n-1}^*+B_{1}z) \right)
\]
is a $\Gamma_{n-1}$-contraction  for all $ z\in \mathbb T$, it
follows from the proof of Theorem \ref{converse} that the Toeplitz operator tuple
$(T_{B_{1}^*+B_{n-1}z},\dots,T_{B_{n-1}^*+B_{1}z},T_z)$ on
$H^2(\mathcal D_{P^*})$ is a $C._0$ $\Gamma_n$-isometry with
$(B_{1},\dots,B_{n-1})$ being the $\ft$-tuple of its adjoint,
$(T_{B_{1}^*+B_{n-1}z}^*,\dots,T_{B_{n-1}^*+B_{1}z}^*,T_z^*)$.
Also it is a $C._0$ $\Gamma_n$-isometry as $T_z$ is a $C._0$ isometry. Again since the Toeplitz operator tuple
$(T_{B_{1}^*+B_{n-1}z},\dots,T_{B_{n-1}^*+B_{1}z},T_z)$ on
$H^2(\mathcal D_{P^*})$ is unitarily equivalent to
$(T_1,\dots,T_2,V)$ on $\mathcal{K}=H^2(\mathbb{D}) \otimes
\mathcal{D}_{P^*}$, $(T_1,\dots,T_{n-1},V)$ is a $C._0$
$\Gamma_n$-isometry. All we have to prove now is that
$(T_1^*,\dots,T_{n-1}^*,V^*)$ is a
$\Gamma_n$-co-isometric extension of $(S_1^*,\dots,S_{n-1}^*,P^*)$.\\

Let us define $
W: \; \mathcal{H} \rightarrow \mathcal{K}$ by $Wh=\sum_{n=0}^{\infty} z^n\otimes D_{P^*}{P^*}^n h$.
Now
\begin{align*}
\|Wh\|^2 &= \|\displaystyle \sum_{n=0}^{\infty}{z^n\otimes
D_{P^*}{P^*}^n h}\|^2 \\&= \langle \displaystyle
\sum_{n=0}^{\infty}{z^n\otimes D_{P^*}{P^*}^n h}\;,\;\displaystyle
\sum_{m=0}^{\infty}{z^m\otimes D_{P^*}{P^*}^m h} \rangle
\\& = \displaystyle \sum_{m,n=0}^{\infty} \langle z^n,z^m \rangle
\langle D_{P^*}{P^*}^nh\;,\;D_{P^*}{P^*}^mh \rangle \\&
=\displaystyle \sum_{n=1}^{\infty}{\langle P^n D_{P^*}^2
{P^*}^nh,h \rangle}\\&= \displaystyle \sum_{n=0}^{\infty}\langle
P^n(I-PP^*){P^*}^nh,h \rangle\\& = \displaystyle
\sum_{n=0}^{\infty}\{\langle P^n{P^*}^nh,h \rangle-\langle
P^{n+1}{P^*}^{n+1}h,h \rangle\} \\&= \|h\|^2-\lim_{n \rightarrow
\infty}\|{P^*}^nh\|^2.
\end{align*}
Since $P$ is a pure contraction, $ \displaystyle
\lim_{n\rightarrow \infty}\|{P^*}^nh\|^2=0$ and thus
$\|Wh\|=\|h\|.$ Therefore $W$ is an isometry. For a basis vector $z^n\otimes \eta$ of $\mathcal{K}$ we have that
\[
\langle W^*(z^n\otimes \eta),h \rangle  = \langle z^n \otimes \eta
, \displaystyle \sum_{k=0}^{\infty}{z^k \otimes D_{P^*}{P^*}^kh}
\rangle  = \langle \eta , D_{P^*}{P^*}^nh \rangle = \langle
P^n D_{P^*}\xi , h \rangle.
\]
Therefore,
\begin{equation}\label{001}
W^*(z^n\otimes \eta)=P^n D_{P^*} \eta, \quad \text{ for }
n=0,1,2,3,\hdots
\end{equation}
and hence
\[
PW^*(z^n \otimes \eta)=P^{n+1} D_{P^*} \eta, \text{ for }
n=0,1,2,3,\hdots.
\]
Again by (\ref{001}),
\[
W^*V(z^n \otimes \eta)=W^*(M_z \otimes I)(z^n \otimes \eta) =
W^*(z^{n+1} \otimes \eta) = P^{n+1}D_{P^*}\eta
=PW^*(z^n\otimes \eta).
\]
Consequently, $W^*V = PW^*$, i.e, $V^*W=WP^*$ and
hence $V^*|_{W(\mathcal H)}=WP^*W^*|_{W(\mathcal H)}$.\\

We now show that $W^*T_i=S_iW^*$ for each $i=1,\dots, n-1$;
\begin{align*}
W^*T_i(z^n \otimes \eta)&=W^*(I\otimes B_{i}^*+M_z\otimes
B_{n-i})(z^n \otimes \eta) \\&=W^*(z^n\otimes B_{i}^*
\eta)+W^*(z^{n+1}\otimes B_{n-i} \eta)\\&=P^nD_{P^*}B_{i}^* \eta
+P^{n+1}D_{P^*}B_{n-1} \eta .
\end{align*}
Also for each $i$, \begin{equation}\label{002} S_iW^*(z^n\otimes
\eta)=S_iP^nD_{P^*} \eta=P^nS_iD_{P^*} \eta . \end{equation}
\noindent \textit{Claim 1.} $S_iD_{P^*}=D_{P^*}B_{i}^*+PD_{P^*}B_{n-i}$.\\
\noindent \textit{Proof of Claim 1.} Since $(B_{1},
\dots,B_{n-1})$ is the $\ft$-tuple of
$(S_1^*,\dots,S_{n-1}^*,P^*)$, we have
\[
(D_{P^*}B_{i}^*+PD_{P^*}B_{n-i})D_{P^*}=(S_i-PS_{n-i}^*)+P(S_{n-i}^*-S_iP^*)=S_iD_{P^*}^2.
\]
Now if $G=S_iD_{P^*}-D_{P^*}B_{i}^*-PD_{P^*}B_{n-i}$, then $G$ is
defined from $\mathcal D_{P^*}$ to $\mathcal H$ and $GD_{P^*}h=0$
for every $h\in \mathcal D_{P^*}$. Thus the claim is proved.\\

So from (\ref{002}) we have
$
S_iW^*(z^n\otimes \eta)=P^n(D_{P^*}B_{i}^*+PD_{P^*}B_{n-i}).
$
Therefore, $W^*T_i=S_iW^*$ and hence $T_i^*|_{W(\mathcal
H)}=WS_i^*W^*|_{W(\mathcal H)}$ for each $i$. Hence the proof is
complete.

\end{proof}

\begin{cor}\label{unitary-dilation}
Let $(S_1,\dots, S_{n-1},P)$ be a $\Gamma_n$-contraction which
satisfies the hypotheses of Theorem \ref{dilation-theorem}. Then
the multiplication operator tuple
$
\left( M_{B_1^*+B_{n-1}z},\dots, M_{B_{n-1}^*+B_1z}, M_z \right )$
on $L^2(\mathcal D_{P^*})$ is a $\Gamma_n$-unitary dilation of $(S_1,\dots, S_{n-1},P)$.
\end{cor}

\begin{proof}
Since the $\Gamma_n$-isometric dilation that is obtained in
Theorem \ref{dilation-theorem} is unitarily equivalent to Toeplitz
operator tuple
$ \left( T_{B_1^*+B_{n-1}z},\dots, T_{B_{n-1}^*+B_1z}, T_z \right )$
on $H^2(\mathcal D_{P^*})$ and since the multiplication operator tuple $\left( M_{B_1^*+B_{n-1}z},\dots, M_{B_{n-1}^*+B_1z}, M_z \right )$ on $L^2(\mathcal D_{P^*})$ is a natural $\Gamma_n$-unitary extension of it, the assertion follows obviously.

\end{proof}

\subsection{Operator models}\label{op-model}

We recall from \cite{nagy} the notion of the characteristic
function of a contraction $T$. For a contraction $T$ defined on a
Hilbert space $\mathcal H$, let $\Lambda_T$ be the set of all
complex numbers for which the operator $I-zT^*$ is invertible. For
$z\in \Lambda_T$, the characteristic function of $T$ is defined as
\begin{eqnarray}\label{e0} \Theta_T(z)=[-T+zD_{T^*}(I-zT^*)^{-1}D_T]|_{\mathcal D_T}.
\end{eqnarray}
By virtue of the relation $TD_T=D_{T^*}P$ (section I.3 of
\cite{nagy}), $\Theta_T(z)$ maps $\mathcal
D_T=\overline{\textup{Ran}}D_T$ into $\mathcal
D_{T^*}=\overline{\textup{Ran}}D_{T^*}$ for every $z$ in
$\Lambda_T$. Let us define
\[
\mathcal{H}_P=(H^2(\mathbb{D})\otimes \mathcal{D}_{P^*}) \ominus
M_{\Theta_P}(H^2(\mathbb{D}) \otimes \mathcal{D}_P).
\]
In \cite{nagy}, Sz.-Nagy and Foias showed that every $C._0$
contraction $P$ defined on a Hilbert space $\mathcal{H}$ is
unitarily equivalent to the operator $P_{\mathcal{H}_P}(M_z
\otimes I_{\mathcal{D}_{P^*}})|_{\mathcal{H}_P}$ on the Hilbert
space $\mathcal{H}_P$, where $P_{\mathcal{H}_P}$ is the projection
of $H^2(\mathbb{D})\otimes \mathcal{D}_{P^*}$ onto
$\mathcal{H}_P$.\\

In this subsection, we independently find a model, which is not necessarily commutative, for a $C._0 \;\Gamma_n$-contraction $(S_1,\dots, S_{n-1},P)$ without assuming any condition on it. On the minimal
dilation space $\mathcal{H}_P$ of $P$ let $R_1,\dots, R_{n-1}$ be
the compressions of $T_1,\dots, T_{n-1}$ (as in Theorem
\ref{dilation-theorem}) respectively to $\mathcal{H}_P$. Evidently
$R_1,\dots, R_{n-1}$ do not commute in general, yet we shall see
that they are individually unitarily equivalent to $S_1,\dots,
S_{n-1}$ respectively. For $P$, we choose the Nagy-Foias model
$R=P_{\mathcal{H}_P}(M_z \otimes
I_{\mathcal{D}_{P^*}})|_{\mathcal{H}_P}$. We also show that under
the hypotheses of Theorem \ref{dilation-theorem}, $(R_1,\dots,
R_{n-1},R)$ on $\mathcal H$ is unitarily equivalent to
$(S_1,\dots, S_{n-1},P)$ and thus in this case it provides a
commutative model. Before going to the main results we state a lemma whose proof could be found in \cite{tirtha-sourav1} (see Lemma 3.3 in \cite{tirtha-sourav1}). For the sake of completeness we include the proof here too.

\begin{lem}\label{L0}
For every contraction $P$, the identity
\begin{eqnarray}\label{L1}
WW^*+M_{\Theta_P}M_{\Theta_P}^*=I_{H^2(\mathbb{D})\otimes
\mathcal{D}_{P^*}}
\end{eqnarray}
holds, where $W$ is the isometry mentioned in the proof of Theorem
\ref{dilation-theorem}.
\end{lem}
\begin{proof}
The operator $W^*$ satisfies the identity
\[
W^*(k_z \otimes \xi)= (I - \bar{z}P)^{-1}D_{P^*}\xi \text{ for } z
\in \mathbb{D} \text{ and } \xi \in \mathcal{D}_{P^*},
\]
where $k_z(w):=(1-\langle w,z\rangle)^{-1}$ for all $w \in
\mathbb{D}$. For a proof one can see Theorem 1.2 in
\cite{arveson2}. Therefore we have
\begin{eqnarray*}
&&\langle (WW^*+M_{\Theta_P}M_{\Theta_P}^*)(k_z \otimes \xi), (k_w
\otimes \eta) \rangle
\\
&=& \langle W^*(k_z \otimes \xi),W^*(k_w \otimes \eta) \rangle +
\langle M_{\Theta_P}^*(k_z \otimes \xi), M_{\Theta_P}^*(k_w
\otimes \eta)  \rangle
\\
&=& \langle (I - \bar{z}P)^{-1}D_{P^*}\xi, (I -
\bar{w}P)^{-1}D_{P^*}\eta  \rangle + \langle k_z \otimes
\Theta_P(z)^*\xi, k_w \otimes \Theta_P(w)^*\eta  \rangle
\\
&=& \langle D_{P^*}(I - wP^*)^{-1}(I - \bar{z}P)^{-1}D_{P^*}\xi,
\eta \rangle + \langle k_z, k_w \rangle \langle
\Theta_P(w)\Theta_P(z)^*\xi, \eta  \rangle
\\
&=& \langle k_z \otimes \xi, k_w \otimes \eta \rangle \text{ for
all $z ,w \in \mathbb{D}$ and $\xi, \eta \in \mathcal{D}_{P^*}$}.
\end{eqnarray*}
The last equality follows from the following well-known
identity
$$
I - \Theta_P(w)\Theta_P(z)^* = (1 - w\bar{z})D_{P^*}(I -
wP^*)^{-1}(I - \bar{z}P)^{-1}D_{P^*}.
$$
Now using the fact that $\{k_z: z \in \mathbb{D}\}$ forms a total
set of $H^2(\mathbb{D})$, the assertion follows.
\end{proof}

\begin{thm}\label{fm}
Let $(S_1,\dots,S_{n-1},P)$ be a $C._0$ $\Gamma_n$-contraction on a
Hilbert space $\mathcal{H}$. Then the operators $S_i$ and $P$ are
unitarily equivalent to $P_{\mathcal{H}_P}(I \otimes B_i^* + M_z
\otimes B_{n-i})|_{\mathcal{H}_P}$ and $P_{\mathcal{H}_P}(M_z
\otimes I_{\mathcal{D}_{P^*}})|_{\mathcal{H}_P}$ respectively,
where $(B_1,\dots,B_{n-1})$ is the $\ft$-tuple of
$(S_1^*,\dots,S_{n-1}^*,P^*)$.
\end{thm}
\begin{proof}
Since $W$ is an isometry, $WW^*$ is the projection onto the range of $W$ and
since $P$ is pure, $M_{\Theta_P}$ is also an isometry. So by Lemma
\ref{L0}, we have that
\[W(\mathcal{H}_P)=(H^2(\mathbb{D})\otimes
\mathcal{D}_{P^*}) \ominus M_{\Theta_P}(H^2(\mathbb{D}) \otimes
\mathcal{D}_P).
\]
So, we have
\begin{eqnarray*}
W^*(I \otimes B_i^* + M_z \otimes B_{n-i})(z^n \otimes \xi)
&=& W^*(z^n \otimes B_i^* \xi) + W^*(z^{n+1} \otimes B_{n-i} \xi)
\\
&=& P^nD_{P^*}B_i^* \xi + P^{n+1}D_{P^*}B_{n-i} \xi
\\
&=& P^n(D_{P^*}B_i^*+PD_{P^*}B_{n-i}) \xi
\\
&=& P^nS_iD_{P^*} \xi \;\;[\text{ by Claim 1 in Theorem
\ref{dilation-theorem}}]
\\
&=& S_iP^nD_{P^*} \xi = S_i W^*(z^n \otimes \xi).
\end{eqnarray*}
So for each $i$, we have $W^*(I \otimes B_i^* + M_z \otimes
B_{n-i})=S_i W^*$ for the vectors of the form $z^n \otimes \xi$,
for all $n \geq 0$ and $\xi \in \mathcal{D}_{P^*}$ which span
$H^2(\mathbb{D})\otimes \mathcal{D}_{P^*}$. Hence we have $W^*(I
\otimes B_i^* + M_z \otimes B_{n-i})=S_i W^*$, which implies that
$W^*(I \otimes B_i^* + M_z \otimes B_{n-i})W=S_i$. Therefore,
$S_i$ is unitarily equivalent to $P_{\mathcal{H}_P}(I \otimes
B_i^* + M_z \otimes B_{n-i})|_{\mathcal{H}_P}$. Again
\begin{eqnarray*}
W^*(M_z \otimes I)(z^n \otimes \xi)=W^*(z^{n+1} \otimes \xi)=
P^{n+1}D_{P^*}\xi= PW^*(z^n \otimes \xi).
\end{eqnarray*}
Therefore by the same argument as above, $P$ is unitarily
equivalent to $P_{\mathcal{H}_P}(M_z \otimes
I_{\mathcal{D}_{P^*}})|_{\mathcal{H}_P}$.
\end{proof}

\begin{cor}\label{representation}
For any $C._0$ $\Gamma_n$-contraction $(S_1,\dots,S_{n-1},P)$,
there are operators $C_1,\dots, C_{n-1}$ satisfying
$\omega(C_i+C_{n-i}z)\leq \binom{n}{i}$ for all $z\in\mathbb T$ and such that $S_i=C_i+PC_{n-i}^*$ for $i=1,\dots, n-1$.
\end{cor}
\begin{proof}
From the previous theorem we have that
\begin{align*} &
W^*(I \otimes B_i^*+M_z \otimes B_{n-i})=S_iW^* \\&\mbox{ or }
W^*(I \otimes B_i^*+M_z \otimes B_{n-i})W=S_i
\\&\mbox{
or } W^*(I \otimes B_i^*)W+W^*(M_z \otimes I)(I \otimes
B_{n-i})W=S_i
\\& \mbox{ or } W^*(I \otimes B_i^*)W+PW^*(I \otimes B_{n-i})W=S_i, \;
\mbox{ since }W^*(M_z \otimes I)=PW^*.
\end{align*}
Taking $C_i= W^*(I \otimes B_i^*)W$ for each $i=1,\dots , n-1$, we
get $S_i=C_i+ PC_{n-i}^*$. The fact that $\omega(C_i+C_{n-i}z)\leq
\binom{n}{i}$ for all $z\in\mathbb T$ is obvious.

\end{proof}

\begin{cor}\label{representation1}
If $(S_1,\dots, S_{n-1},P)$ is a $\Gamma_n$-contraction with
$\|P\|<1$, then there are $n-1$ unique operators
$C_1,\dots,C_{n-1}$ such that $S_i=C_i+PC_{n-i}^*$.
\end{cor}
\begin{proof}
Let there be $C_1,\dots, C_{n-1}$ and $\tilde{C}_1,\dots,
\tilde{C}_{n-1}$ such that $S_i=C_i+PC_{n-i}^*$ and
$S_i=\tilde{C}_i+P\tilde{C}_{n-i}^*$. Then we have
$D_i+PD_{n-i}^*=0$, where $D_i=C_i-\tilde{C}_{n-i}$. Now
$$\|D_i\|=\|-PD_{n-i}^*\|\leq \|P\|\|D_{n-i}\|< \|D_{n-i}\| \quad \mbox{as }
\|P\|<1.$$ Since this holds for each $i$, replacing $i$ by $n-i$
we get $\| D_{n-i} \|< \| D_{i} \|$. This shows that $D_i=0$ for each
$i$ and consequently $C_i=\tilde{C}_i$.
\end{proof}

\begin{thm}\label{modelthm}
Let $(S_1,\dots,S_{n-1},P)$ be a $C._0$ $\Gamma_n$-contraction on
a Hilbert space $\mathcal{H}$ that satisfies the hypotheses of
Theorem \ref{dilation-theorem}. Then $(S_1,\dots,S_{n-1},P)$ is
unitarily equivalent to the tuple $(R_1,\dots,R_{n-1},R)$ on the
Hilbert space $\mathcal H_P=(H^2(\mathbb D)\otimes \mathcal
D_{P^*})\ominus M_{\Theta_P}(H^2(\mathbb D)\otimes \mathcal D_P)$
defined by
\begin{align*}
& R_i=P_{\mathcal H_P}(I\otimes B_{i}^*+M_z\otimes
B_{n-i})|_{\mathcal H_P},\; \text{ for } i=1,\dots,n-1 \\&
\text{and } R=P_{\mathcal H_P}(M_z\otimes I)|_{\mathcal H_P}.
\end{align*}
\end{thm}

\begin{proof}
Following the proof of Theorem \ref{dilation-theorem}, it suffices
to show that $W(\mathcal H)=\mathcal H_P$, which follows from the
fact that
$
WW^*+M_{\Theta_P}M_{\Theta_P}^*=I_{H^2(\mathbb D)\otimes \mathcal D_{P^*}}$. Hence the proof is complete.
\end{proof}

\vspace{0.4cm}

\section{A complete unitary invariant for $C._0$
$\Gamma_n$-contractions}

\vspace{0.4cm}

\noindent A complete unitary invariant for a class of operator tuples, defined on a Hilbert space $\mathcal H$, is a necessary and sufficient condition under which any two such operator tuples are unitarily equivalent, that is, there is a unitary on $\mathcal H$ that intertwines the corresponding components of the two tuples. For a pair of contractions $P, P'$ acting on Hilbert spaces $\mathcal{H}$ and $\mathcal{H'}$ respectively, we say that the
characteristic functions of $P$ and $P'$ coincide if there are
unitary operators $u: \mathcal{D}_P \to \mathcal{D}_{P'}$ and
$u_{*}: \mathcal{D}_{P^*} \to \mathcal{D}_{{P'}^*}$ such that the
following diagram commutes for all $z \in \mathbb{D}$,

$$
\begin{CD}
\mathcal{D}_P @>\Theta_P(z)>> \mathcal{D}_{P^*}\\
@Vu VV @VVu_{*} V\\
\mathcal{D}_{P'} @>>\Theta_{{P'}}(z)> \mathcal{D}_{{P'}^*}
\end{CD}.
$$
A few decades ago, Sz.-Nagy and Foias proved the following
theorem, which asserts that the characteristic function
is a complete unitary invariant for the c.n.u contractions.

\begin{thm}[Nagy-Foias, \cite{nagy}]\label{nf}
Two completely non-unitary contractions are unitarily equivalent
if and only if their characteristic functions coincide.
\end{thm}

Here we present a complete unitary invariant for the $C._0$
$\Gamma_n$-contractions.

\begin{thm}\label{unitary inv}
 Let $(S_1,\dots,S_{n-1},P)$ and $(S_1',\dots,S_{n-1}',P')$ be two $C._0$ $\Gamma_n$-contractions
 defined on $\mathcal{H}$ and $\mathcal{H'}$ respectively.
 Suppose $(B_1,\dots,B_{n-1})$ and $(B_1',\dots,B_{n-1}')$ are the $\ft$-tuples of
 $(S_1^*,\dots,S_{n-1}^*,P^*)$ and $(S_1'^*,\dots,S_{n-1}'^*,P'^*)$ respectively.
 Then $(S_1,\dots,S_{n-1},P)$ is unitarily equivalent to $(S_1',\dots,S_{n-1}',P')$ if and only if
 $(B_1,\dots,B_{n-1},\Theta_P)$ and
 $(B_1',\dots,B_{n-1}',\Theta_{P'})$ are unitarily equivalent.
 \end{thm}

\begin{rem}
The sense carried out by the unitary equivalence of
\[
(B_1,\dots,B_{n-1},\Theta_P) \text{ and }
 (B_1',\dots,B_{n-1}',\Theta_{P'})
\]
 is that the characteristic
 functions of $P$ and $P'$ coincide and $(B_1,\dots,B_{n-1})$ is
 unitarily equivalent to $(B_1',\dots,B_{n-1}')$
 by the unitary $u_{*}: \mathcal{D}_{P^*} \to
\mathcal{D}_{{P'}^*}$ that is involved in the coincidence of the
characteristic functions of $P$ and $P'$.
\end{rem}

\begin{proof}

First let us assume that $(S_1,\dots,S_{n-1},P)$ and
$(S_1',\dots,S_{n-1}',P')$ are unitarily equivalent and let
$U:\mathcal{H} \to \mathcal{H'}$ be a unitary such that
$US_i=S_i'U$ for each $i$ and $UP=P'U$. Since $P$ is a $C._0$
contraction, it is a completely non-unitary contraction and hence
by Theorem \ref{nf}, the characteristic functions of $P$ and $P'$
coincides. The unitary $u_{*}: \mathcal{D}_{P^*} \to
\mathcal{D}_{{P'}^*}$ that is involved in the coincidence of the
characteristic functions $\Theta_P$ and $\Theta_{P'}$ is nothing
but the restriction of $U$ to $\mathcal D_{P^*}$ that takes
$\mathcal D_{P^*}$ onto $\mathcal D_{{P'}^*}$. An interested
reader can see Chapter VI of \cite{nagy} for a proof of this fact.
We now prove that the same unitary intertwines the $\mathcal
F_O$-tuples of $(S_1^*,\dots,S_{n-1}^*,P^*)$ and
${(S_1'}^*,\dots,{S_{n-1}'}^*,{P'}^*)$. We have
$$
UD_{P^*}^2=U(I-PP^*)=U-{P'}{P'}^*U=D_{{P'}^*}^2U,
$$
which gives $UD_{P^*}=D_{{P'}^*}U$. Let
$\tilde{U}=U|_{\mathcal{D}_{P^*}}$. Then note that $\tilde{U}\in
\mathcal{B}(\mathcal{D}_{P^*}, \mathcal{D}_{{P'}^*})$ and
$\tilde{U}D_{P^*}=D_{{P'}^*}\tilde{U}$. Now for each $i$,
\begin{eqnarray*}
D_{{P'}^*}\tilde{U}B_i\tilde{U}^*D_{{P'}^*}=\tilde{U}D_{P^*}B_iD_{P^*}\tilde{U}^*
&=& \tilde{U}(S_i^*-S_{n-i}P^*){\tilde{U}}^*
\\
&=& S_i'^*-S_{n-i}'P'^*=D_{{P'}^*}B_i'D_{{P'}^*}.
\end{eqnarray*}
Therefore we have $\tilde{U}B_i\tilde{U}^*=B_i'$ for each
$i=1,\dots,n-1$.\\

We prove the converse part. Let $u: \mathcal{D}_P \to
\mathcal{D}_{P'}$ and $u_{*}: \mathcal{D}_{P^*} \to
\mathcal{D}_{{P'}^*}$ be unitary operators such that for each $i$
\[
u_{*}B_i=B_i'u_{*} \text{ and } u_{*} \Theta_P(z) = \Theta_{P'}(z)
u \text{ for all $z \in \mathbb{D}$}.
\]
The unitary operator $u_{*}: \mathcal{D}_{P^*} \to
\mathcal{D}_{{P'}^*}$ induces the following unitary operator
\begin{gather*}
U_{*}: H^2(\mathbb{D})\otimes \mathcal{D}_{P^*} \to
H^2(\mathbb{D})\otimes \mathcal{D}_{P'^*} \\
\quad \quad \quad \quad \quad (z^n \otimes \xi) \mapsto (z^n
\otimes u_{*}\xi) \xi \in \mathcal{D}_{P^*} \;,\; n \geq 0.
\end{gather*}
 We note here that
\[
U_{*}(M_{\Theta_P}f(z))=u_{*}\Theta_{P}(z)f(z)=\Theta_{P'}(z)uf(z)=M_{\Theta_{P'}}(uf(z)),
\]
for all $f \in H^2(\mathbb{D})\otimes \mathcal{D}_{P}$ and $z \in
\mathbb{D}$. Hence $U_{*}$ takes ${\rm Ran}\, M_{\Theta_P}$ onto
${\rm Ran}\, M_{\Theta_{P'}}$. Since $U_{*}$ is unitary, we have
\[
U_{*}(\mathcal{H}_P) =
U_{*}(({\rm Ran} \, M_{\Theta_P})^\bot)=(U_{*}{\rm Ran}\, M_{\Theta_P})^{\bot}=({\rm Ran} \, M_{\Theta_{P'}})^\bot=\mathcal{H}_{P'}.
\]
Again from the definition of $U_{*}$ we have that
\begin{eqnarray*}
U_{*}(I \otimes B_i^* + M_z \otimes B_{n-i})^*&=&(I \otimes
u_{*})(I \otimes B_i + M_z^* \otimes B_{n-i}^*)
\\
&=& I \otimes u_{*} B_i + M_z^* \otimes u_{*} B_{n-i}^*
\\
&=& I \otimes B_i'u_{*} + M_z^* \otimes B_{n-i}'^*u_{*}
\\
&=& (I \otimes B_i' + M_z^* \otimes B_{n-i}'^*)(I \otimes u_{*}) \\
&=&(I \otimes B_i'^* + M_z \otimes B_{n-i}')^*U_{*}.
\end{eqnarray*}
Therefore, $\mathcal{H}_{P'}=U_{*}(\mathcal{H}_P)$ is a joint
co-invariant subspace of $(I \otimes B_i'^* + M_z \otimes
B_{n-i}')$ for $i=1,\dots, n-1$. Hence $ P_{\mathcal{H}_P}(I
\otimes B_i^* + M_z \otimes B_{n-i})|_{\mathcal{H}_P}$ and
$P_{\mathcal{H}_{P'}}(I \otimes B_i'^* + M_z \otimes
B_{n-i}')|_{\mathcal{H}_{P'}} $ are unitarily equivalent for each
$i$. It is evident that the unitary operator that intertwines them
is $U_{*}|_{\mathcal{H}_P}:\mathcal{H}_P \to \mathcal{H}_{P'}$.
Also we have
\[
U_{*}(M_z \otimes I_{\mathcal{D}_{P^*}})=(I \otimes u_{*})(M_z
\otimes I_{\mathcal{D}_{P^*}})=(M_z \otimes
I_{\mathcal{D}_{P'^*}})(I \otimes u_{*})=(M_z \otimes
I_{\mathcal{D}_{P'^*}})U_{*}.
\]
So $P_{\mathcal{H}_P}(M_z \otimes
I_{\mathcal{D}_{P^*}})|_{\mathcal{H}_P}$ and $
P_{\mathcal{H}_{P'}}(M_z \otimes
I_{\mathcal{D}_{P'^*}})|_{\mathcal{H}_{P'}}$ are unitarily
equivalent by the same unitary $U_{*}|_{\mathcal{H}_P}$. Therefore
$(S_1,\dots,S_{n-1},P)$ and $(S_1',\dots,S_{n-1}',P')$ are
unitarily equivalent and the proof is complete.

\end{proof}

\end{document}